\documentclass[12pt]{amsart}
\usepackage{verbatim}
\usepackage{float}
\usepackage{array}
\usepackage{longtable}

\usepackage{footnotehyper}
\makesavenoteenv{table}
\makesavenoteenv{tabular}

\usepackage{amssymb}
\usepackage{nicematrix, tikz}
\usetikzlibrary{fit}
\usepackage{enumitem}

\newcommand{\lradG}{0.25}
\newcommand{\darkradE}{0.115}

\numberwithin{equation}{section}

\newtheorem{thm}[equation]{Theorem}
\newtheorem{prop}[equation]{Proposition}

\newtheorem{lem}[equation]{Lemma}
\newtheorem{cor}[equation]{Corollary}

\theoremstyle{definition}
\newtheorem{rem}[equation]{Remark}

\newtheorem{ntt}[equation]{}

\newcommand{\Mot}{\mathcal{M}}

\newcommand{\U}{\mathcal{U}}

\newcommand{\SB}{\mathop{\mathrm{SB}}}

\newcommand{\CH}{\mathop{\mathrm{CH}}\nolimits}

\newcommand{\Spin}{\operatorname{\mathrm{Spin}}}

\newcommand{\SU}{\operatorname{\mathrm{SU}}}

\newcommand{\E}{\mathrm{E}}
\newcommand{\ff}{\mathbb{F}}

\newcommand{\Ch}{\mathop{\mathrm{Ch}}\nolimits}

\newcommand{\cores}{\mathop{\mathrm{cor}}\nolimits}

\newcommand{\zz}{\mathbb{Z}}

\newcommand{\F}{\mathrm{F}}
\newcommand{\G}{\mathrm{G}}
\newcommand{\A}{\mathrm{A}}

\newcommand{\qq}{\mathbb{Q}}

\newcommand{\C}{\mathrm{C}}

\newcommand{\Spec}{\operatorname{Spec}}

\newcommand{\D}{\mathrm{D}}

\title[Motives, cohomological invariants and the Freudenthal magic square]{Motives, cohomological invariants and the Freudenthal magic square}


\author[Nikita Geldhauser]{Nikita Geldhauser}

\author[Alexander Henke]{Alexander Henke}

\author[Maksim Zhykhovich]{Maksim Zhykhovich}

\address{Geldhauser: Mathematisches Institut der Universit\"at M\"unchen, Germany}
\email{geldhauser@math.lmu.de}

\address{Henke: Independent Researcher, Wiesbaden, Germany}
\email{eniszden@freenet.de}

\address{Zhykhovich: Mathematisches Institut der Universit\"at M\"unchen, Germany}
\email{zhykhovich@math.lmu.de}

\thanks{The work of the first and of the third author is supported by the DFG research grants SE 1721/4-1 and ZH 918/2-1.}

\begin{document}

\begin{abstract}
We investigate cohomological invariants and motivic invariants of semisimple algebraic groups arising in the Freudenthal magic square. Besides, we show that if the Rost invariant of a strongly inner group of type $\E_7$ is a sum of at most two symbols modulo $2$, then it is isotropic over an odd degree field extension, and use this fact to give a different proof of a result of Petrov and Rigby. Moreover, we give a motivic interpretation of a result of Garibaldi and Petersson about a cohomological invariant of degree $5$ for certain groups of type $^2\E_6$ which detects their isotropy.
\end{abstract}

\maketitle

\newsavebox{\dEpic}
\savebox{\dEpic}(5,1){\begin{picture}(5, 1)
    \multiput(3,0.75)(1,0){2}{\circle*{\darkradE}}
    \multiput(3,0.25)(1,0){2}{\circle*{\darkradE}}
    \multiput(1,0.5)(1,0){2}{\circle*{\darkradE}}
    
    \put(1, 0.5){\line(1,0){1}}
    \put(3,0.75){\line(1,0){1}}
    \put(3, 0.25){\line(1,0){1}}
    
    \put(3,0.5){\oval(2,0.5)[l]}
    \end{picture}}

\section{Introduction}

The Freudenthal magic square was discovered in the 1950s independently by Jacques Tits and Hans Freudenthal. Since then, the topic has been the subject of numerous articles in algebra, geometry, and theoretical physics.

The Freudenthal magic square contains algebraic groups or Lie algebras given by the celebrated Tits construction (see \cite[\S41]{KMRT}).
In the present article we describe a new symmetry of the Freudenthal magic square, which is related to cohomological invariants of algebraic groups (see Table~\ref{tab1} in Section~\ref{sec8}).

The results of this article  emerge from the interplay of cohomological invariants, notably the Rost invariant and Tits algebras, of algebraic groups and the motivic $J$-invariant.

The idea to use cohomological invariants in the classification of algebraic groups goes back to Jean-Pierre Serre. In particular, Serre conjectured the existence of an invariant of degree $3$ for groups of type $\F_4$ and $\E_8$. This invariant was later constructed by Markus Rost for all $G$-torsors, where $G$ is a simple simply connected algebraic group, and is now called the Rost invariant.

The motivic $J$-invariant was introduced by Vishik in \cite{Vi05} for quadratic forms, in \cite{PSZ08} for arbitrary semisimple algebraic groups of inner type, in \cite{Fi19} for Hermitian forms, and in \cite{GZ26} for semisimple groups of outer type. This invariant encodes the structure of Chow motives of generically split twisted flag varieties.

The $J$-invariant was an important tool to solve several long-standing problems. For example, it plays an important role in the progress on the Kaplansky problem about possible values of the $u$-invariant of fields by Vishik \cite{Vi07},  in the solution of a problem of Serre about groups of type $\E_8$ and its finite subgroups \cite{S16}, \cite{GS10}, and in the classification of generically split twisted flag varieties in \cite{PS10}, \cite{PS12}.

\medskip

In the present article we use motivic techniques and give a motivic interpretation of a degree $5$ cohomological invariant due to Garibaldi and Petersson \cite{GaPe07} of groups of type $^2\E_6$ which split over a quadratic field extension and whose Rost invariant is a pure symbol modulo $2$ (see Section~\ref{sec4}). It was shown in \cite{GP24} that such groups arise from a Tits construction and thus fit into the Freudenthal magic square.

Note that analogous results on groups of type $\F_4$, $\E_7$ and $\E_8$ were established earlier in \cite[Theorem~3.12]{MD09}, \cite[Chapter~12]{He23} and \cite[Theorem~8.7]{S16} (see also Section~\ref{sec8} below).

In the course of our proofs we obtain motivic decompositions of some twisted flag varieties of type $^2\E_6$ (Proposition~\ref{t211pro}) and relate the Rost invariant and the $J$-invariant (Proposition~\ref{lp24}). Moreover, we classify generically quasi-split varieties of type $^2\E_6$ (see Remark~\ref{genqs}).
This result is analogous to the classification of generically split varieties given in \cite{PS10} and \cite{PS12} for groups of inner type.

\medskip

In Section~\ref{sec5} we investigate the properties of the Rost invariant of groups of type $^2\E_6$ and $\E_7$. 
In particular, we give a cohomological criterion for the isotropy of groups of type $\E_7$ up to odd degree field extensions (see Corollary~\ref{cor52}). Recall (see \cite{GaPS16}) that Markus Rost and Tonny Springer independently raised the general question of possible relations between the Rost invariant and the rationality of parabolic subgroups for groups of type $\E_7$. This question was settled in \cite{GaPS16}. Corollary~\ref{cor52} improves results of \cite{GaPS16}.

Furthermore, we use Corollary~\ref{cor52} to give a different proof of a result of Petrov and Rigby \cite{PR22} saying that over $2$-special fields the Tits construction cannot produce groups of type $\E_8$ with the semisimple anisotropic kernel of type $\E_7$ (see Corollary~\ref{peri}).

\medskip

In Section~\ref{secc5} we briefly recall an Allison--Faulkner construction of algebraic groups. This construction has an advantage that it is more symmetric and allows to ignore an apparent asymmetry of the Tits constructions.

We illustrate this in Section~\ref{secc5} considering a degree $5$ invariant of groups of type $\E_7$. Note that groups of type $\E_7$ are the biggest groups from the Freudenthal magic square possessing two different Tits constructions.

\medskip

Section~\ref{sec7} is devoted to possible values of the $J$-invariant of groups of type $^2\E_6$ and $\E_7$ (see Proposition~\ref{j110} and Proposition~\ref{prop76}). The problem of determining possible values of the $J$-invariant has a long history starting from an explicit question of Vishik \cite[Question~5.13]{Vi05} in the context of quadratic forms. For representative examples we refer to \cite[last column of Table in Section~4.13]{PSZ08} for groups in general, to \cite[Corollary~8.10 and Corollary~10.4]{GaPS16} for groups of type $^1\E_6$ and to \cite[Proposition~8.24]{PS21} for groups of type $\E_8$.

\medskip

Finally, in Section~\ref{sec8} we summarize the results of this article related to groups from the Freudenthal magic square.
It turns out that under some additional conditions (specified in Section~\ref{sec8}) these groups possess a cohomological invariant of degree $2$, $3$, $4$ or $5$ (see Table~\ref{tab1}). This yields an additional symmetry of the Freudenthal magic square that, apparently, has not been discussed in the literature before.

\medskip

\noindent
{\bf Acknowledgments}. We sincerely thank Skip Garibaldi for his valuable comments on exceptional algebraic groups.

\section{Preliminaries on algebraic groups and motives}

\begin{ntt}[Algebraic groups]
For the background of the theory of algebraic groups we refer to books \cite{Bor}, \cite{Spr}, and \cite{KMRT}.

In the present article we freely use the notion of a Tits index (see \cite{Ti66}), Tits algebras (see \cite[\S27.A]{KMRT}), Rost invariant (see \cite{GaMS03}), Rost multiplier (see \cite[p.~122]{GaMS03}), concepts related to Jordan algebras, e.g. Albert and Freudenthal algebras (see \cite[Chapter~IX]{KMRT}), Tits constructions (see \cite[\S41]{KMRT}), and the Freudenthal magic square.

Let $G$ be a semisimple linear algebraic group of rank $n$ over a field $F$. We say that the group $G$ is {\it strongly inner}, if $G$ is a group of inner type and all its Tits algebras are split.

For a subset $I\subset\{1,2,\ldots, n\}$ we denote by $X_I$ the variety of parabolic subgroups of $G$ of type $I$. According to our convention, the variety of Borel subgroups is $X_{1,2,\ldots,n}$ and the varieties of maximal parabolic subgroups of $G$ are denoted by $X_1, X_2,\ldots, X_n$. The enumeration of simple roots follows Bourbaki.

For a prime number $p$ we say that $G$ is {\it $p$-anisotropic}, if it is anisotropic over a $p$-special closure of $F$.

When talking about cohomological invariants of algebraic groups modulo $2$, we assume that the characteristic of $F$ is different from $2$. For the Tits constructions we assume that the characteristic of $F$ is different from $2$ and $3$. Finally, talking about degree $5$ invariants of groups of type $\E_7$ and $\E_8$ we assume that the characteristic of $F$ is $0$.
\end{ntt}

\begin{ntt}[Chow rings and motives]
For the background of the intersection theory and motives we refer to books \cite{Ful} and \cite{EKM}.

For a smooth variety $X$ over $F$ we denote by $\CH^*(X)$ its Chow ring. For a prime number $p$ we denote by $\Ch^*(X)$ the Chow ring of $X$ modulo $p$.

We denote
by $\cores$ the corestriction of a variety or of a cycle (see \cite[Section~3]{Ka10a}, \cite[Section~3]{GZ26})

In the present article we work in the category of Chow motives with $\ff_p$-coefficients. We write $\Mot(X)$ for the Chow motive of a smooth projective variety $X$. For a motive $M$ and a non-negative integer $l$ we denote by $M\{l\}$ its Tate twist by $l$.

For a smooth projective variety $X$ over $F$ and a field extension $L/F$ we write $X_L$ for the respective extension of scalars.

For a twisted flag variety $X$ we denote by $\U(X)$ its upper motive (see \cite[Section~1]{Ka10a}).

In the case of groups of outer type $^2\E_6$ we sometimes use the category of (co)normed Chow motives (see \cite[Section~5]{GZ26}). For a twisted flag variety $X$, a direct summand $M$ of the motive of $X$ or a cycle $\alpha$ we denote by 
$\overline X$, $\overline M$ and $\bar\alpha$ the extension
of scalars either to a splitting field (when we work in the category of ordinary Chow motives) or to a quasi-splitting field (when we work on the category of conormed Chow motives: see \cite[End of Section~5]{GZ26}) of $X$.

Finally, we denote by
$P(M,t)\in\zz[t]$ (resp. by $P_N(M,t)\in\zz[t]$) the Poincar\'e polynomial (resp. the conormed Poincar\'e polynomial) of the motive $M$ over a splitting (resp. quasi-splitting) field. Note that we use the conormed Poincar\'e polynomials only in Section~\ref{sec4} for  groups of type $^2\E_6$.

\end{ntt}

\begin{ntt}[$J$-invariant]
In the present article we intensively use the notion of the $J$-invariant of semisimple algebraic groups (see \cite{PSZ08} and \cite{GZ26}).

The $J$-invariant is a discrete invariant which describes the
motivic behavior of the variety of Borel subgroups of a semisimple linear algebraic group. 
Let $G$ be a {\it split} semisimple group over a field $F$, let $B$ be a Borel subgroup of $G$, let $E$ be a $G$-torsor, and let $p$ be a prime number. 
It turns out that the Chow motive of the twisted flag variety $E/B$ with coefficients in $\ff_p$ decomposes into a direct sum of Tate twists of an indecomposable motive $R_p(E)$ and the Poincar\'e polynomial of $R_p(E)$ over a splitting field of $E$ is a product of ``cyclotomic-like'' polynomials. More precisely, it equals 

\begin{equation}\label{f11}
\prod_{i=1}^r\frac{t^{d_ip^{j_i}}-1}{t^{d_i}-1}\in\zz[t],
\end{equation}
for some non-negative integers $r$, $d_i$ and $j_i$. The parameters $r$ and $d_i$ are combinatorial and are given in a table in \cite[Section~4.13]{PSZ08}, and the $r$-tuple
$(j_1,\ldots,j_r)$ is of geometric nature and constitutes the $J$-invariant of $E$. Moreover, there are explicit combinatorial upper bounds for the entries of the $J$-invariant: $j_i\le k_i$, which are also given in \cite[Section~4.13]{PSZ08}.

Note also that the parameters $d_i$ and $k_i$ are contained
essentially in a table of Kac \cite[Table~2]{Kac85}.

Moreover, there is an analogous concept of the $J$-invariant for groups of outer type, i.e., when $G$ is a {\it quasi-split} semisimple group, given in \cite{GZ26}. In this case one works in the category of conormed Chow motives instead of the ordinary ones (see \cite[Section~5]{GZ26}). The respective table with the parameters of the $J$-invariant is given in \cite[Table~5]{GZ26}.

Sometimes we will include the degrees $d_i$ to the $J$-invariant and write
$\underset{d_1\ \ \ldots\  \ d_r}{(j_1,\ldots,j_r)}$ instead of $(j_1,\ldots,j_r)$.

The following table contains the maximal values of the $J$-invariant for the prime $p=2$ for groups from the Freundenthal magic square. Note that for groups of inner type we work in the category of ordinary Chow motives and for groups of outer type we work in the category of conormed Chow motives. Moreover, we omit the cases $\A_1$, $^2\A_2$, $\C_3$ and $\F_4$, since we do not use their $J$-invariants in the present article.

Note that the maximal values of the $J$-invariants can be realized over some fields, but the respective groups do not necessarily arise from the Tits constructions. For example,
the maximal possible $J$-invariant for groups of type $\E_7$ arising from the $\A_1\times \F_4$ Tits construction is $(1,1,0,0)$. 

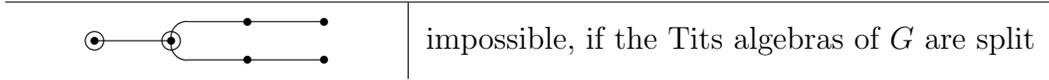
\begin{longtable}{c|c|c|c|c|c}
\caption{Maximal values of the $J$-invariant}\label{tab:Jmax}\\

$2\,{^2\A_2}$ & $^2\A_5$ & $^1\D_6$ & $^2\E_6$ & $\E_7$ & $\E_8$\\
\hline
\endfirsthead

$2\,{^2\A_2}$ & $^2\A_5$ & $^1\D_6$ & $^2\E_6$ & $\E_7$ & $\E_8$\\
\hline
\endhead

$\displaystyle\underset{3\ \,3}{(1,1)}$
&
$\displaystyle\underset{2\ \,1\ \,3\ \,5}{(0,1,1,1)}$
&
$\displaystyle\underset{1\ \,1\ \,3\ \,5}{(1,3,1,1)}$
&
$\displaystyle\underset{3\ \,5\ \,9}{(1,1,1)}$
&
$\displaystyle\underset{1\ \,3\ \,5\ \,9}{(1,1,1,1)}$
&
$\displaystyle\underset{3\ \,5\ \,9\ \,15}{(3,2,1,1)}$
\\

\end{longtable}
\end{ntt}

\section{Motives of $^2\E_6$-varieties and a cohomological invariant of degree $5$}\label{sec4}

Let $G$ be an algebraic group of type $^2\E_6$ over a field $F$ of characteristic different from~$2$.
We denote by $K$ the respective quadratic field extension of $F$. 

If the Tits algebras of $G$ are split, then by \cite[Lemma~2.1]{GaPe07} there is a well-defined Rost invariant of $G$, which we denote by $r(G)$. In general, one has $12r(G)=0$ (see \cite[p.~442]{KMRT}).

Note, however, that if the group $G$ splits over $K$, then the Tits algebras of $G$ are automatically split and $2r(G)=0$.

\begin{prop}\label{t21pro}
Let $G$ be an algebraic group of type $^2\E_6$ over a field $F$ of characteristic different from $2$ and let $K$ be the respective quadratic field extension of $F$.

Assume that the group $G$ splits over $K$ and the Rost invariant $r(G)$ of $G$ is a pure symbol in $H^3(F,\mu_2)$.

Then there exists a functorial cohomological invariant $u\in H^5(F,\mu_2)$ of $G$ such that for every field extension $L/F$
the variety $(X_{2})_L$ has a rational point if and only if $u_L=0$.
Moreover, $u$ is a pure symbol.
\end{prop}
\begin{proof}
This is a combination of \cite[Theorem~3.3]{GP24} and \cite[Theorem~0.2, equivalence of (1) and (4)]{GaPe07}.
\end{proof}

\begin{rem}
For example, all conditions of Proposition~\ref{t21pro} are satisfied by the (unique) anisotropic simply connected group of type $^2\E_6$ over the field $\mathbb{R}$.    
\end{rem}

In general, it has been conjectured that there is a correspondence between pure symbols in $H^n(F,\mu_2)$ and binary motives of dimension $2^{n-1}-1$ (see, e.g., \cite[Introduction]{S16}).

The next proposition confirms this in the $^2\E_6$-case as above. 
Namely, we give a motivic interpretation of the invariant of Proposition~\ref{t21pro} and construct the respective binary motive as a direct summand of the motive of the variety $X_{2}$. Note that in the proof we use neither the main theorem of Garibaldi and Petersson \cite[Theorem~0.2]{GaPe07} nor any result from \cite{GP24}.

\begin{prop}\label{t211pro}
Let $G$ be an algebraic group of type $^2\E_6$ over a field $F$ of characteristic different from $2$
and let $K$ be the respective quadratic extension of $F$. Assume that the group $G$ splits over $K$ and the Rost invariant $r(G)$ of $G$
is a pure symbol in $H^3(F,\mu_2)$.

Assume additionally that the variety $X_{2}$ does not have zero-cycles of odd degree.

Then we have the following decomposition in the category of Chow motives with $\ff_2$-coefficients:
\begin{equation}\label{eee}
\Mot(X_2)\simeq
\U(X_2)\oplus\U(X_2)\{6\}\oplus
(\oplus_{i\in I}R_2(G)\{i\}))\oplus_{j\in J}(\cores_{K/F}(\Spec K)\{j\}),
\end{equation}
where
${I=\{1,2,6,7,8,10,11,12,16,17\},\,J=[1,\ldots,18]\coprod [7,\ldots,14]}$.

The Poincar\'e polynomial of $\U(X_2)$ equals $1+t^{15}$
and $R_2(G)\simeq\U(X_{1,2,3,4,5,6})$ is the Rost motive corresponding to the pure symbol $r(G)$.
\end{prop}
\begin{proof}
The proof consists of several steps.

{\it Step 1.} Analysis of the Tits indices.

First, we copy the table \cite[Proposition~2.3]{GaPe07} describing the Rost invariants for all Tits indices (see \cite{Ti66}) of {\it isotropic} groups of type $^2\E_6$ with trivial Tits algebras. Note that for isotropic groups of type $^2\E_6$ we always have that $2r(G)=0$.

\begin{longtable}{c|c}
\caption{Tits indices and the Rost invariants}\label{tabb2}\\
Tits index & Condition on the Rost invariant \\ \hline
quasi-split & $r(G) = 0$ \\
\parbox{5cm}{\begin{picture}(5, 1.1)
\put(0,0){\usebox{\dEpic}} 
\put(4,0.5){\oval(0.4,0.75)}
\put(1,0.5){\circle{\lradG}}
\end{picture}}& $r(G)$ is a non-zero pure symbol divisible by $K$\\
\parbox{5cm}{\begin{picture}(5, 1.1)
\put(0,0){\usebox{\dEpic}} 
\put(4,0.5){\oval(0.4,0.75)}
\end{picture}}& $r(G)$ is a pure symbol not divisible by $K$\\
\parbox{5cm}{\begin{picture}(5, 1)
\put(0,0){\usebox{\dEpic}} 
\put(1,0.5){\circle{\lradG}}
\end{picture}}& $r(G)$ is not a pure symbol\\
\hline
\parbox{5cm}{\begin{picture}(5, 1)
\put(0,0){\usebox{\dEpic}} 
\put(1,0.5){\circle{\lradG}}
\put(2,0.5){\circle{\lradG}}
\end{picture}} & impossible, if the Tits algebras of $G$ are split
\end{longtable}

In particular, since by our assumptions the Rost invariant of $G$ is a pure symbol divisible by $K$, it follows immediately that for every field extension $L/F$ the varieties $(X_2)_L$, $(X_{1,6})_L$, and $(X_{1,2,6})_L$ are isotropic resp. anisotropic simultaneously.

In this case it is well known that the upper motives of $X_2$, $X_{1,6}$, and $X_{1,2,6}$ are isomorphic over $F$.

{\it Step 2.} Possible motivic summands of $X_2$ and $X_{1,6}$.

First, we cite a theorem of Karpenko about motivic summands of projective homogeneous varieties of outer type.

Let $G$ be a semisimple algebraic group over $F$.
Let $K/F$ be a Galois field extension such that the group $G_K$ is of inner type and assume that $|K:F|=p^n$ for some $n$ and some prime $p$.

\begin{prop}[Karpenko, {\cite[Theorem~1.1]{Ka10a}}]\label{pro32k}
Under above assumptions
let $X$ be a twisted flag variety over $F$ homogeneous under a semisimple algebraic group $G$.

The complete motivic decomposition
of $X$ over $F$ consists of Tate twists of upper motives of the form $\cores_{E/F}\U(Y)$, where $E$ runs over all intermediate subfields $F\subseteq E\subseteq K$ and $Y$ runs over all twisted flag varieties over $E$ homogeneous under the action of $G_E$.
\end{prop}

\begin{rem}
Karpenko gives in \cite{Ka10a} a more precise statement about possible motivic summands. The above statement will be sufficient for the purposes of the present article.
\end{rem}

According to Karpenko possible indecomposable summands of the Chow motive of $X_{2}$ (and of $X_{1,6}$) considered with $\ff_2$-coefficients are Tate twists of motives from the following list (a posteriori not all motives from this list may appear):
$$\U(X_{2}),\, \U(X_{1,2,3,4,5,6}),\text{ and }\cores_{K/F}(\Spec K).$$

Indeed, since $G$ splits over $K$, the only corestriction $\cores_{K/F}$ which may appear is $\cores_{K/F}(\Spec K)$. Moreover, we have by Step~1 that $\U(X_{1,6})\simeq\U(X_{1,2,6})\simeq\U(X_{2})$, and, therefore, no other upper motives appear in the list.

{\it Step 3.} The upper motives of generically quasi-split varieties of type $^2\E_6$.

\begin{prop}\label{lpr25}
Let $G$ be an arbitrary group of type $^2\E_6$ over an arbitrary field.
Assume that $J(G)\ne (0,0,0)$. Then the varieties $X_{1,6}$, $X_2$ and $X_{1,2,6}$ are not generically quasi-split.
\end{prop}
\begin{proof}
Note first that it is a general fact that all twisted flag varieties homogeneous under an action of a group $G$ of outer type are never generically split, i.e., the group $G$ can not become split over their function fields. Thus, we can require at most quasi-split instead of split.

Let $J(G)=(j_1,j_2,j_3)$. Since $J(G)$ is not trivial and $1\ge j_1\ge j_2\ge j_3\ge 0$ (see \cite[Table~5]{GZ26}), we have $j_1=1$.

Assume that $X_{1,6}$ or $X_2$ is generically quasi-split. Then their conormed Chow motives are direct sums of Tate twists of the same indecomposable motive with Poincar\'e polynomial $(1+t^3)(1+t^5)^{j_2}(1+t^9)^{j_3}$.
But the conormed Poincar\'e polynomials of $X_2$ and $X_{1,6}$ are not divisible by $1+t^3$ in the semiring $\mathbb{N}_0[t]$. Indeed,
a direct computation (see \cite[Proposition~6.1]{GZ26}) shows that

$$P_N(X_2,t)=\frac{(t^8-1)(t^{12}-1)(t^9+1)}{(t-1)(t^4-1)(t^3+1)}\text{ and}$$

$$P_N(X_{1,6},t)=\frac{(t^8-1)(t^{12}-1)(t^5+1)(t^9+1)}{(t-1)(t+1)(t^4-1)(t^4+1)}$$
and both polynomials are not divisible by $1+t^3$ in $\mathbb{N}_0[t]$ (although they are divisible by $1+t^3$ in $\zz[t]$). Contradiction.

Assume now that the variety $X_{1,2,6}$ is generically quasi-split. Note that the Poincar\'e polynomial $P_N(X_{1,2,6},t)$ is divisible by $1+t^3$ in $\mathbb{N}_0[t]$. Therefore, we will argue in a slightly different way.

If $X_{1,2,6}$ is generically quasi-split, then the product $X_2\times X_{1,6}$ is generically quasi-split as well. But the Poincar\'e polynomial $$P_N(X_2\times X_{1,6},t)=P_N(X_{2},t)P_N(X_{1,6},t)$$ is not divisible by $1+t^3$ in $\mathbb{N}_0[t]$. This finishes the proof exactly in the same way as above.
\end{proof}

\begin{rem}\label{genqs}
It follows from Proposition~\ref{lpr25} and from Table~\ref{tabb2} that if the Tits algebras of $G$ are split, then all other twisted flag $G$-varieties not listed in Proposition~\ref{lpr25} are generically quasi-split. Moreover, if $J(G)=(0,0,0)$, then all twisted flag $G$-varieties are generically quasi-split.

If the Tits algebras of $G$ are not split, then one should distinguish more cases. The varieties $X_2$, $X_4$ and $X_{2,4}$ are never generically quasi-split by the index reduction formula.
If ${J(G)=(0,0,0)}$, then all twisted flag $G$-varieties apart from $X_2$, $X_4$ and $X_{2,4}$ are generically quasi-split.

If $J(G)\ne(0,0,0)$, then the varieties $X_{1,6}$, $X_2$, $X_{1,2,6}$, $X_4$, $X_{2,4}$ are not generically quasi-split and all other twisted flag varieties are so.
\end{rem}

\begin{prop}\label{lp24}
Let $G$ be a simply connected algebraic group of type $^2\E_6$ with trivial Tits algebras. Then the following conditions are equivalent:

1) For the Rost invariant of $G$ we have $6r(G)=0$ and $3r(G)$ is a non-zero pure symbol.

2) The $J$-invariant of $G$ equals $(1,0,0)$.

Moreover, if these conditions are satisfied, then the upper motive of the variety of Borel subgroups of $G$ modulo $2$ is the Rost motive corresponding to the Rost invariant of $G$.
\end{prop}
\begin{proof}
It is well known that there is an embedding of $G$ in a simply connected strongly inner group of type $\E_7$ with Rost multiplier $1$. Moreover, both groups have the same $J$-invariant (see \cite[Section~7, case $^2\E_6$]{GZ26}).

But for groups of type $\E_7$ the equivalence of 1) and 2) is established in \cite[Table~10.B]{GaPS16}.

For the last part of the lemma it suffices to show that for an arbitrary field extension $L/F$ the group $G_L$ is quasi-split (up to odd degree field extensions) if and only if the Pfister quadric $Q$ corresponding to the Rost invariant $3r(G)$ is isotropic over $L$.

If $G_L$ is quasi-split (up to odd degree field extensions), then it is straightforward that $Q_L$ is isotropic.

For the converse implication it suffices to show that $J(G_L)=(0,0,0)$. We assume the contrary. Then by Proposition~\ref{lpr25} the group $G_{F(X_2)}$ is isotropic, but not quasi-split. On the other hand, this is impossible by Table~\ref{tabb2}, if $3r(G)_L=0$.
\end{proof}

\begin{rem}
In Proposition~\ref{lp24} one can omit the condition that the Tits algebras of $G$ are trivial. In this case instead of the Rost invariant one should take the invariant $b(G)$ defined in \cite[Lemma~3 and Definition~4]{DCG17} and assume that $2b(G)=0$ and $b(G)$ is a non-zero pure symbol.
\end{rem}

Thus, we have determined the $J$-invariant for our group $G$ and, in particular, the Poincar\'e polynomial of $\U(X_{1,2,3,4,5,6})$, namely $1+t^3\in\zz[t]$. Note that in the case $^2\E_6$ the conormed upper motive of the variety of Borel subgroups has the same Poincar\'e polynomial as the ordinary upper motive (see \cite[Section~8, Exceptional types]{GZ26}).

\begin{rem}
In the proofs of Proposition~\ref{lpr25} and Proposition~\ref{lp24} we do not use that $G$ splits over $K$.
\end{rem}

{\it Step 4.} Motives of isotropic varieties of type $^2\E_6$.

We would like to describe the motives of $X_2$ and $X_{1,6}$, when our group $G$ of type $^2\E_6$ has anisotropic kernel of type $^2\A_3$.

In this case we use the Chernousov--Gille--Merkurjev--Brosnan method. In \cite{CGM05} and \cite{Br05} they described an algorithm to compute motives of isotropic varieties based on the structure of double cosets for the Weyl group. In our situation we are interested in the double cosets $W(\A_3)\!\setminus\! W(\E_6)/W(\A_5)$ for $X_2$ and $W(\A_3)\!\setminus\! W(\E_6)/W(\D_4)$ for $X_{1,6}$. For example, the Tate motives in the resulting motivic decompositions correspond precisely to double cosets which are invariant under the $*$-action and have cardinality $1$.

Note that the method of Chernousov--Gille--Merkurjev--Brosnan reduces the motives of isotropic varieties of type $^2\E_6$ to the motives of twisted flag
varieties for the semisimple anisotropic kernel, i.e., to the case of $^2\A_3$-varieties.
Note also that all twisted flag varieties for $^2\A_3$ (with trivial Tits algebras) are generically quasi-split (this follows, for example, from the Tits diagrams above).

In fact, we will not need a precise formula for the motives, but just the following partial information about them.
We have by the Chernousov--Gille--Merkurjev--Brosnan method:

\begin{align}\label{f27}
\begin{split}
\Mot(X_2)&\simeq\oplus_{i=0,6,15,21}\ff_2\{i\}\oplus N_1\\
\Mot(X_{1,6})&\simeq\oplus_{i=0,9,15,24}\ff_2\{i\}\oplus N_2
\end{split}
\end{align}
where $N_1$ and $N_2$ are direct sums of corestrictions of the form $\cores_{K/F}(\Spec K)\{i\}$ and of the motives of homogeneous $^2\A_3$-varieties corresponding to the anisotropic kernel of $G$.
(Note that $\dim X_2=21$ and $\dim X_{1,6}=24$, so there is a symmetry in the above decompositions).

Since every $^2\A_3$-homogeneous variety is generically quasi-split, the only possible indecomposable motivic summands for these varieties are by Proposition~\ref{pro32k} either $\U(X_{1,2,3,4,5,6})$ or $\cores_{K/F}(\Spec K)$ (up to Tate twists).
Therefore, the motives $N_1$ and $N_2$ are direct sums of motives of at most these two types with positive Tate twists.

{\it Step 5.} Motives of anisotropic varieties of type $^2\E_6$.

In this step we assume that the variety $X_2$ (and hence $X_{1,6}$) does not have zero-cycles of odd degree.

\begin{prop}[De Clercq, {\cite[Theorem~1.1]{DC13}}]\label{declercq}
Let $X$ and $Y$ be two twisted flag varieties for
semisimple algebraic groups, let $M$ be a direct summand of the motive of $Y$, and let $N$ be the upper motive of
$X$. Assume that a Tate twist $(N\{i\})_{F(Y)}$ is an indecomposable direct summand of $M_{F(Y)}$ and $Y$ has an $F(X)$-rational point. Then $N\{i\}$ is a direct summand of $M$ over~$F$.
\end{prop}

In the context of this proposition we will use the following terminology. We will say that if a direct summand $N\{i\}$ is visible over $F(Y)$, then it is already visible over $F$.

We apply this proposition in the following case:
$X=X_{1,2,3,4,5,6}$, $Y=X_2$ or $Y=X_{1,6}$, $M=\Mot(Y)$ and $N=\U(X_{1,2,3,4,5,6})$.

It follows from Proposition~\ref{lpr25} and Step~1 that the anisotropic kernel of $G$ over $F(Y)$ is of type $^2\A_3$. In particular, the motive $N$ remains indecomposable over $F(Y)$. Thus, we obtain motivic decompositions of $Y$ over $F(Y)$ as in formulae~\eqref{f27} in Step~4.

Now it follows from De Clercq's proposition that all direct summands of $N_1$ and $N_2$ isomorphic to $\U(X_{1,2,3,4,5,6})$ (up to Tate twist) are already visible over $F$.

We have the following general lemma.

\begin{lem}
Let $K/F$ be a quadratic field extension, let $Y$ be a twisted flag variety for a semisimple algebraic group over $F$, and let $M$ be a motive over $F$ modulo $2$. Assume that $Y$ has a $K$-rational point and that $\cores_{K/F}(\Spec K)\{i\}_{F(Y)}$ is an indecomposable direct summand of $M$ over $F(Y)$.

Then $\cores_{K/F}(\Spec K)\{i\}$  is a direct summand of $M$ over $F$.
\end{lem}
\begin{proof} This is a particular case of \cite[Proposition 4.1]{Ka12a} (applied with $S=\cores_{K/F}(\Spec K) $ and $E=F(Y)$).
\end{proof}

We apply this lemma in the following situation: $Y$ is the variety $X_2$ or $X_{1,6}$ and $M=\Mot(Y)$ is the motive $Y$.
The conditions of the lemma are satisfied, since by our assumptions the group $G$ splits over $K$.

It follows then that all direct summands of $N_1$ and $N_2$ isomorphic to $\cores_{K/F}(\Spec K)$ (up to Tate twist) over $F(Y)$ are already visible over $F$.
Therefore, the motives $N_1$ and $N_2$ are defined over $F$, and the remaining summands in decompositions~\eqref{f27} belong to the upper motives of $X_2$ and $X_{1,6}$.

But by Step~1 the upper motives of $X_2$ and $X_{1,6}$ are isomorphic. Moreover, since we assumed at the beginning that the variety $X_{1,6}$ does not have zero-cycles of odd degree, the only possibility for the upper motive of $X_{1,6}$ (and $X_2$) is a generically split binary motive with Poincar\'e polynomial $1+t^{15}$ (the intersection of $\{0,6,15,21\}$ and $\{0,9,15,24\}$).

In summary, we have following motivic decompositions of $X_{1,6}$ and $X_2$ over $F$:
$$\Mot(X_2)\simeq
\U(X_2)\oplus\U(X_2)\{6\}\oplus
N_1\text { and}$$
$$\Mot(X_{1,6})\simeq
\U(X_{1,6})\oplus\U(X_{1,6})\{9\}\oplus N_2$$
where $N_1$ and $N_2$ are direct sums of positive Tate twists of $\U(X_{1,2,3,4,5,6})$ and $\cores_{K/F}(\Spec K)$, and the Poincar\'e polynomial of $\U(X_2)\simeq\U(X_{1,6})$ equals $1+t^{15}$.

Note also that over every field extension $L/F$ the motive $\U(X_2)\simeq\U(X_{1,6})$ is split if and only if $X_{1,6}$ (or $X_2$) has a zero-cycle over $L$ of odd degree.

Finally, we have for the Poincar\'e polynomial of $X_2$ (see \cite[Proposition~6.1]{GZ26}):

$$P(X_2,t)=\frac{(t^8-1)(t^{9}-1)(t^{12}-1)}{(t-1)(t^3-1)(t^4-1)}.$$

Comparing $P(X_2,t)$ with the conormed Poincar\'e polynomial $P_N(X_2,t)$ we immediately can recover all direct summands of $N_1$ isomorphic to Tate twists of $\cores_{K/F}(\Spec K)$. Since by the above considerations the remaining direct summands of $N_1$ are $\U(X_{1,2,3,4,5,6})$ (up to Tate twists), we can compute the exact Tate twists in decomposition~\eqref{eee} explicitly.
\end{proof}

\section{On the Rost invariant of groups of type $^2\E_6$ and $\E_7$}\label{sec5}

Let $r(G)$ denote the Rost invariant of a group $G$ with trivial Tits algebras of type $^2\E_6$ or $\E_7$ over a field $F$ of characteristic different from $2$. Since the Rost multiplier of the embedding $^2\E_6\to \E_7$ of quasi-split simply connected groups equals $1$, the respective groups of types $^2\E_6$ and $\E_7$ have the same Rost invariant.

\begin{thm}\label{rost2symbth}
Assume that $6r(G)=0$ and $3r(G)$ is a sum of two symbols. Then $3r(G)$ is a sum of two symbols with a common slot.
\end{thm}
\begin{proof}
By assumptions we can consider the Rost invariant as an element in $${H^3(F,\mu_2)\simeq I^3/I^4},$$ where $I$ denotes the fundamental ideal of the Witt ring of $F$.

Since the Rost invariant is a sum of two symbols, we can represent it modulo $I^4$ by a sum of two $3$-fold Pfister forms. Since we are working modulo $I^4$ we can replace the sum of two Pfister forms by a difference of two Pfister forms. In particular, we can represent the Rost invariant by a $14$-dimensional quadratic form $q$ from $I^3$ modulo $I^4$.

Let $X$ denote the variety of Borel subgroups of the group $G$ and let $Y$ denote the variety of Borel subgroups of a group of type $\D_7$ defined by the quadratic form $q$.

Passing to the function field $F(Y)$ of $Y$ we make the form $q$ hyperbolic and the Rost invariant of $G$ trivial. Therefore, by \cite[Theorem~0.5]{Ga01a} the group $G$ splits over $F(Y)$ and the variety $X$ becomes isotropic.

Conversely, passing to the function field $F(X)$ of $X$ we split the group $G$. The Rost invariant becomes zero, and, therefore, the form $q$ becomes zero modulo $I^4$. But by the Arason--Pfister Hauptsatz all non-zero anisotropic forms from $I^4$ have dimension at least $2^4=16$. Since $q$ is $14$-dimensional, it splits completely over $F(X)$, and, thus, the variety $Y$ becomes isotropic.

It follows now that the upper motives of $X$ and $Y$ are isomorphic. These upper motives are determined by the values of the $J$-invariants of the respective groups. In particular, the group $G$ and the quadratic form $q$ have the same $J$-invariants. For groups of type $\D_7$ we have that the $J$-invariant is of the form $(j_1,j_2)$ with $d_1=3$, $d_2=5$, $j_1\le 2$, $j_2\le 1$ (note that since $q\in I^3$, the respective group of type $\D_7$ is strongly inner and we take the parameters of the $J$-invariant for $\Spin_{14}$).

On the other hand, the $J$-invariant of $G$ has entries of degrees $3$, $5$, and $9$. Therefore, its entry of degree $9$ must vanish. But it follows then from \cite[Table~10.B]{GaPS16} that $G$ is isotropic with anisotropic kernel of type $\D_6$ (or smaller: $\D_4$ or trivial). But then we are reduced to a $12$-dimensional quadratic form from $I^3$, where it is well known that there is a common slot.
\end{proof}

This theorem allows to formulate an isotropy criterion for groups of type $\E_7$ in a more optimal way. Recall (see \cite{GaPS16}) that Markus Rost and Tonny Springer independently raised the general question of possible relations between the Rost invariant and the rationality of parabolic subgroups for groups of type $\E_7$. This question was settled in \cite{GaPS16}.

\begin{cor}\label{cor52}
Let $G$ be a group of type $\E_7$ over a field $F$ with trivial Tits algebras. Then $G$ is isotropic over an odd degree field extension of $F$ if and only if $6r(G)=0$ and $3r(G)$ is a sum of at most two symbols.
\end{cor}
\begin{proof}
This is a combination of
Theorem~\ref{rost2symbth}
and \cite[Theorem~10.18]{GaPS16}.
\end{proof}

Now we can give a different proof of the following result of Petrov and Rigby (see \cite[Theorem~4.4]{PR22}). The original proof used the structure of Cartan symmetric spaces for groups of type $\E_8$ and many explicit computations related to Lie algebras of type $\E_8$. Our approach does not use these techniques.

\begin{cor}[Petrov--Rigby]\label{peri}
Let $F$ be a field which has no non-trivial field extensions of odd degree. Let $G$ be a group of type $\E_8$ arising from the Tits construction. Then the semisimple anisotropic kernel of $G$ cannot be of type $\E_7$.
\end{cor}
\begin{proof}
It is well known that the Rost invariant of groups of type $\E_8$ arising from the Tits construction has symbol length at most $2$ (see \cite[Section~11.6]{Ga09}).

Assume that the semisimple anisotropic kernel $H$ of $G$ is of type $\E_7$. Then the Rost invariant of $H$ also has symbol length at most $2$ and, thus, by Corollary~\ref{cor52} is isotropic. Contradiction.
\end{proof}

\begin{rem} Analogous statements are known for other Tits constructions.

The semisimple anisotropic kernel for groups of type $\E_6$ arising from the $\mu_2\times \F_4$ Tits construction is never of type $^2\A_5$ (see \cite{GaPe07}).

For groups of type $\E_7$ arising from the $\A_1\times\F_4$ Tits construction the semisimple anisotropic kernel is never of type $\D_6$ (see \cite{He23}).
\end{rem}

\begin{rem}
In \cite[p.~60]{Ti66} Tits added a question mark to groups of type $\E_8$ with the semisimple anisotropic kernel of type $\E_7$. Later in \cite{Ti90} Tits proved the existence of such groups over the field $\qq_p(t)$.
Note that by \cite[Corollary~10.17]{GaPS16} the group constructed by Tits has the Rost invariant whose even component is a $4$-torsion, but not a $2$-torsion, and this explains the complexity of the proof of Tits \cite{Ti90}.

On the other hand, there is a simplified proof of the existence of this kind of groups of type $\E_8$ using the $J$-invariant. It was given by Petrov and Stavrova in \cite[Section~8]{PSt11} and their method is uniform for groups of all types.
\end{rem}

\section{Groups arising from the Allison--Faulkner construction}\label{secc5}

In this section we briefly recall a construction of algebraic groups and Lie algebras introduced in \cite{AF93}, which we call the Allison--Faulkner construction.

The Allison--Faulkner construction was extensively examined in \cite{R22} and \cite{PR22}. The construction takes a so called {\it structurable algebra} $(A,-)$ and an element $\gamma \in (F^\times/F^{\times 2})^3$ as input. In contrast to the famous Tits construction it is more symmetric.
Over $2$-special fields all outputs of the Tits construction can also be obtained by the Allison--Faulkner construction.

In the following table we list some structurable algebras and conditions which are needed to produce the same groups as the ones from the Tits constructions. We omit the $\F_4$ case and write $\mathcal{Q}$, $\mathcal{O}$, $\mathcal{A}$ for quaternion, octonion and Albert algebras. For a scalar $a\in F^\times$ we write $(a)$ for the respective \'etale quadratic algebra over $F$ or for its class in $H^1(F,\mu_2)$.

A bioctonion\footnote{Another name of them is octooctonions.} algebra is a twisted form of the tensor product of two octonion algebras. Bioctonion algebras are classified by $H^1(F, (\G_2 \times \G_2) \rtimes \mu_2)$ (see \cite[Table~5]{R22}). If the outer action of $\mu_2$ is trivial, we speak of a decomposable bioctonion algebra.

\begin{longtable}{c|c|c|c}
\caption{Exceptional groups via Tits and Allison--Faulkner constructions}
\label{tab:tits-af}\\

Group & Tits construction & Input AF construction & Condition\\
\hline
\endfirsthead

Group & Tits construction & Input AF construction & Condition\\
\hline
\endhead

$\E_6$
& $\mu_2\times\F_4$ or $\G_2\times\A_2$
& $(a)\otimes\mathcal{O}\times\gamma$
& \\

$\E_7$
& $\A_1\times\F_4$ or $\G_2\times\C_3$
& $\mathcal{Q}\otimes\mathcal{O}\times\gamma$
& \\

$\E_8$
& $\G_2\times\F_4$
& $\text{bioctonions}\times\gamma$
& Decomposable bioctonions \\
\end{longtable}

The input of the Allison--Faulkner construction makes it more visible that in the case of the two Tits constructions of $\E_7$ (or $\E_6$), basically the same information is used as input.
In one case the octonion algebra plus $\gamma$ gives the Albert algebra via construction of an (exceptional) Jordan algebra (see \cite[Section~37.C]{KMRT}), which yields the $\A_1\times\F_4$ Tits construction.
In the other case the quaternion algebra plus $\gamma$ gives a degree six central simple algebra with a symplectic involution, which yields the $\G_2\times\C_3$ Tits construction.

In \cite[Chapter~12]{He23} Henke constructed a cohomological invariant $h_5\in H^5(F,\mu_2)$ of degree $5$ of groups $G$ of type $\E_7$ with the $J$-invariant $(1,0,0,0)$. This invariant is zero if and only if the variety of maximal parabolic subgroups of $G$ of type $1$ has a zero-cycle of odd degree.

Let now $G$ be a group of type $\E_7$ arising from the $\A_1\times \F_4$ Tits construction. The input of this construction consists of a quaternion algebra $\mathcal Q$ and an Albert algebra $\mathcal A$. If we assume that the Albert algebra $\mathcal A$ splits over the function field of the Severi--Brauer variety $\SB(\mathcal Q)$, then the $J$-invariant of $G$ will be $(1,0,0,0)$. Henke shows in \cite[Theorem~12.3.6]{He23} that for such groups the invariant $h_5$ coincides with the $f_5$-invariant of the Albert algebra $\mathcal A$.

The following proposition makes a connection with the $\G_2\times\C_3$ Tits construction of groups of type $\E_7$. The input of this construction consists of an octonion algebra and a central simple algebra of degree $6$ with a symplectic involution.
 
\begin{prop}\label{prop61}
The Henke $h_5$-invariant of groups of type $\E_7$ is defined for the output of the $\G_2\times\C_3$ Tits construction, if the respective octonion algebra splits over the function field of the Severi--Brauer variety $\SB(\mathcal{Q})$, where $\mathcal Q$ is the quaternion algebra corresponding to $\C_3$.
 \end{prop}
\begin{proof}
Let $G$ be the output of the Tits construction.
It is well known that $\mathcal Q$ is Brauer equivalent to the Tits algebra of $G$.

It follows from the assumptions that the group $G$ splits over $\SB(\mathcal{Q})$. Therefore, by \cite[Lemma~7.1.9]{He23} the $J$-invariant of $G$ is $(1,0,0,0)$ or $(0,0,0,0)$. In both cases the invariant $h_5$ is defined by \cite[Proposition~12.3.1]{He23}.
\end{proof}

\section{On the $J$-invariant of groups of type $^2\E_6$ and $\E_7$}\label{sec7}

It follows from Proposition~\ref{lp24} that the $J$-invariant of a compact real group of type $^2\E_6$ is $(1,0,0)$. On the other hand, a generic group of type $^2\E_6$ has a maximal value of the $J$-invariant, namely $(1,1,1)$. Note also that for the entries $(j_1,j_2,j_3)$ of the $J$-invariant we always have $1\ge j_1\ge j_2\ge j_3\ge 0$.

In this section we will show that there exist anisotropic groups of type $^2\E_6$ with the remaining possible $J$-invariant $(1,1,0)$. This case corresponds precisely to the situation of Theorem~\ref{rost2symbth}.

\begin{prop}\label{j110}
Let $G$ be a simply connected group of type $^2\E_6$ with trivial Tits algebras over a field $F$ with the maximal $J$-invariant $(1,1,1)$. Let $H$ be the respective strongly inner group of type $\E_7$, where $G\le H$ has the Rost multiplier $1$, and let $Y_1$ be the variety of maximal parabolic subgroups of $H$ of type $1$. Then the group $G$ is $2$-anisotropic over $F(Y_1)$
and has the $J$-invariant $(1,1,0)$.
\end{prop}
\begin{proof}
Since the $J$-invariant of $G$ equals $(1,1,1)$ over $F$, the group $G$ is anisotropic and remains anisotropic over every odd degree field extension of $F$.

We pass to the function field $F(X_2)$, where $X_2$ is the variety of maximal parabolic subgroups of $G$ of type $2$. In this case we reduce to the anisotropic kernel of type $^2\A_5$ or less. But analyzing the $J$-invariant for groups of type $^2\A_5$ (see \cite[Table~5]{GZ26}) we see that the $J$-invariant becomes $(1,1,0)$ or less. Therefore, by \cite[Table~10.B]{GaPS16} the variety $Y_1$ has a zero-cycle of odd degree over $F(X_2)$. (In fact, one can argue in a more direct way: The group $G$ is obviously isotropic over $F(X_2)$. Therefore, the bigger group $H$ is isotropic as well. The classification of Tits indices for strongly inner groups of type $\E_7$ by coincidence implies that the variety $Y_1$ is isotropic). 

Assume now that passing to the function field $F(Y_1)$ makes the group $G$ isotropic. The $J$-invariant over $F(Y_1)$ becomes $(1,1,0)$ by \cite[Lemma~10.9 and Table~10B]{GaPS16} and the Rost invariant becomes a sum of two symbols. Therefore, by the table in Step~1 of the proof of Proposition~\ref{t211pro} the anisotropic kernel of $G$ becomes of type $^2\A_5$ and the variety $X_2$ becomes isotropic.

It follows now that the upper motives of $X_2$ and $Y_1$ are isomorphic over $F$. But $\dim X_2=21$ which is strictly less than $\dim\,\U(Y_1)=\dim Y_1=33$ (see \cite[Proposition~10.20]{GaPS16}). We have arrived to a contradiction.
\end{proof}

\begin{rem}
Note that opposite to the $^2\E_6$-case the groups of type $\E_7$ with the $J$-invariants $(0,1,0,0)$ and $(0,1,1,0)$ are always isotropic over a field extension of odd degree.

Such situation is not common. For example, for analogous embeddings of simply connected groups $^3\D_4\to \F_4$ for $p=3$ and $\SU(W,h)\to\Spin(V,q)$ (i.e., $^2\A_{n-1}\to\D_n$) for $p=2$ both groups are simultaneously $p$-isotropic or $p$-anisotropic (and in the latter case one can even give a precise correspondence between the Witt indices of respective Hermitian and quadratic forms); see \cite[Proof of Proposition~8.1]{GZ26} and \cite[Lemma~9.1]{Ka12b}.
\end{rem}

\begin{rem}
Using similar arguments one can show that the group $G$ of Proposition~\ref{j110} remains anisotropic even over the field $F(Y_7)$. Over this field the $J$-invariant of $G$ equals $(1,0,0)$.
\end{rem}

\begin{rem}
Proposition 10.20 of \cite{GaPS16} states that the variety $Y_1$ is incompressible, if the $J$-invariant equals $(0,1,1,1)$. This is sufficient for the proof of Proposition~\ref{j110}. On the other hand, in \cite[Theorem~9.3.10]{He23} Henke showed using a coaction introduced in \cite{PS21} a more precise statement, namely that in this case
$$\Mot(Y_1)\simeq N\oplus\Big(\oplus_{i=2}^{14} R_2(G)\{i\}\Big),$$
where $R_2(G)$ is the upper motive of the variety of Borel subgroups of $G$ (its Poincar\'e polynomial equals $(1+t^3)(1+t^5)(1+t^9)$) and $N$ is an indecomposable motive with the Poincar\'e polynomial $(1+t^9)(1 + t + t^4 + t^6 + t^8 + t^{12} + t^{16} + t^{18} + t^{20} + t^{23} + t^{24})$.
\end{rem}

Now we investigate possible $J$-invariants for anisotropic adjoint groups $G$ of type $\E_7$.

Recall that the $J$-invariant of $G$ equals $(j_1,j_2,j_3,j_4)$ with $1\ge j_1\ge 0$, $1\ge j_2\ge j_3 \ge j_4\ge 0$.
If the group $G$ is strongly inner, then $j_1=0$ and the only possible $J$-invariant 
for a $2$-anisotropic $G$ is $(0,1,1,1)$.

The compact real group of type $\E_7$ has the $J$-invariant $(1,0,0,0)$.  Moreover, a generic group of type $\E_7$ has the $J$-invariant $(1,1,1,1)$. We show now that the $J$-invariant $(1,1,0,0)$ is also possible for certain groups $G$.

\begin{rem}
It seems to be unknown whether there exist anisotropic groups of type $\E_7$ with the $J$-invariant $(1,1,1,0)$.

As pointed out in \cite{He23}, every anisotropic group of type $\E_7$, which has the semisimple anisotropic kernel of type $\D_6$ over the generic point of the Severi--Brauer variety of its Tits algebra, has this $J$-invariant. Unfortunately, we are not aware whether such groups exist.
\end{rem}

\begin{prop}\label{prop76}
There exist $2$-anisotropic groups of type $\E_7$ with the $J$-invariant $(1,1,0,0)$. 
\end{prop}
\begin{proof}
Let $G$ be a group of type $\E_7$ obtained from the $\A_1\times\F_4$ Tits construction. Let $\mathcal Q$ be the respective quaternion algebra corresponding to $\A_1$. Then
by \cite[Lemma~11.1.5]{He23} the group $G$ over the function field $F(\SB(\mathcal{Q}))$ is either split or has the semisimple anisotropic kernel of type $\D_4$. In the latter case the $J$-invariant of $G_{F(\SB(\mathcal{Q}))}$ equals $(0,1,0,0)$. Therefore, by \cite[Theorem~4.1]{Zh24} the $J$-invariant of $G$ is at most $(1,1,0,0)$.

Let now $G$ be the output of a {\it generic} $\A_1\times\F_4$ Tits construction. The group $G$ is anisotropic, since we can specialize to the field of real numbers, where the output of the Tits construction is anisotropic.

Moreover, we can specialize $G$ to the field of real numbers, where the output of the Tits construction has the semisimple anisotropic kernel of type $\D_4$. Thus, the $J$-invariant of $G$ is at least $(1,1,0,0)$. Combining this with the considerations in the beginning of the proof we conclude that the $J$-invariant of $G$ is precisely $(1,1,0,0)$.
\end{proof}

\begin{rem}
We can look at the Killing form of the respective Allison--Faulkner construction of groups of type $\E_7$. It  
takes a quaternion algebra $\mathcal{Q}$, an octonion algebra $\mathcal{O}$ and $\gamma = (\gamma_1,\gamma_2,\gamma_3) \in (F^\times/F^{\times 2})^3$ as input.

We denote by $n_\mathcal{Q}$, $n_\mathcal{O}$ the norm forms of the respective algebras and let $n'_\mathcal{Q}, n'_\mathcal{O}$ denote their pure parts. Then by \cite[Theorem~28.1]{R22} the Killing form of the output $G$ is given by
 \begin{equation}
\mathcal{K}_G = \langle -1 \rangle (4n'_\mathcal{O} \perp 3\langle 2 \rangle n'_\mathcal{Q} \perp \langle \gamma_1\gamma^{-1}_2,\gamma_2\gamma^{-1}_3,\gamma_3\gamma^{-1}_1 \rangle n_\mathcal{O}n_\mathcal{Q}).
 \end{equation}

In particular, in general, the Killing form is not split over $F(\SB(\mathcal{Q}))$ and, therefore, the $J$-invariant of the respective group is $(1,1,0,0)$.
\end{rem}

\section{Freudenthal magic square}\label{sec8}

In this section we consider the Freudenthal magic square (see \cite[p.~540]{KMRT}). This is a table which contains the celebrated Tits constructions of algebraic groups and Lie algebras. We reproduce the Freudenthal magic square in Table~\ref{tab1}.

Under some additional conditions, which we specify in Table~\ref{tab3}, each of the groups from the Freudenthal magic square has a cohomological invariant of degree $2$, $3$, $4$ or $5$. These invariants measure whether the variety of parabolic subgroups of certain type (specified in Table~\ref{tab2}) has a zero-cycle of odd degree.

Table~\ref{tab2} contains the same groups from the Freudenthal magic square, but is independent of the Tits constructions. We list there general conditions on groups in motivic terms (column 3) and in equivalent algebraic terms (column 4), when they possess cohomological invariants $f_n\in H^n(F,\mu_2)$ of respective degrees. The column on the right contains the parabolic subgroup $P_\Theta$ giving a binary motive. The binary motive has dimension $2^{n-1}-1$, where $n$ is the degree of the respective cohomological invariant from the second column.

\newpage

\begin{longtable}{@{}c c@{}}
\caption{Freudenthal magic square and degrees of cohomological invariants of the respective groups}
\label{tab1}\\

\begin{tabular}{r|cccc}
  & $\A_1$ & $^2\A_2$ & $\C_3$ & $\F_4$ \\
\hline
base field                          & $\A_1$   & $^2\A_2$   & $\C_3$   & $\F_4$ \\
quadratic extension ($\mu_2$)       & $^2\A_2$ & $2\,^2\A_2$& $^2\A_5$ & $^2\E_6$ \\
quaternion algebra ($\A_1$)         & $\C_3$   & $^2\A_5$   & $^1\D_6$ & $\E_7$ \\
octonion algebra ($\G_2$)           & $\F_4$   & $^2\E_6$   & $\E_7$   & $\E_8$ \\
\end{tabular}
&
\begin{tabular}{c|cccc}
   & \multicolumn{4}{c}{ } \\
   \hline
   & 2 & 3 & 4 & 5 \\
   & 3 & 3 & 4 & 5 \\
   & 4 & 4 & 4 & 5 \\
   & 5 & 5 & 5 & 5 \\
\end{tabular}
\end{longtable}

\begin{longtable}{m{1.2cm}|m{1.2cm}|m{5cm}|m{5cm}|m{1.2cm}}
\caption{Cohomological invariants of algebraic groups from the Freudenthal magic square}\label{tab2} \\
 \centering Group    & Degree of $f_n$ & \centering Conditions  & Equivalent conditions up to odd degree field extensions &  $P_\Theta$ \\
  \hline
  $\A_1$ & $2$ & no conditions & no conditions & any  \\
  \hline
  $^2\A_2$ & $3$ & no conditions & no conditions & $P_{1,2}$  \\
  \hline
  $2\,{^2\A_2}$ & $3$ & $J=\underset{3\ \,3}{(1,0)}$ & the group is a product of two equal groups of type $^2\A_2$ & any  \\
 \hline
  $\C_3$ & $4$ & no conditions & no conditions & $P_{2}$  \\
  \hline
  $^2\A_5$ & $4$ & the group splits over a quadratic field extension and $J=\underset{2\ \,1\ \,3\ \,5}{(0,1,0,0)}$ & the group splits over a quadratic field extension and over the function field of its Tits algebra & $P_{1,5}$  \\
 \hline
  $^1\D_6$ & $4$ &  ${J=\underset{1\ \,1\ \,3\ \,5}{(0,1,0,0)}}$ & the group corresponds to an excellent $12$-dimensional quadratic form & $P_{1}$  \\
 \hline
 $\F_4$ & $5$ & no conditions & no conditions & $P_{4}$  \\
 \hline
 $^2\E_6$ & $5$ & the group splits over a quadratic field extension and $J=\underset{3\ \,5\ \,9}{(1,0,0)}$ & the group splits over a quadratic field extension and the mod $2$ component of the Rost invariant is a pure symbol & $P_{2}$  \\
 \hline
$\E_7$  & $5$ & $J=\underset{1\ \,3\ \,5\ \,9}{(1,0,0,0)}$ & the group splits over the function field of its Tits algebra & $P_1$ \\

 \hline
   $\E_8$  & $5$ & $J=\underset{3\ \,5\ \,9\ \,15}{(0,0,0,1)}$ & the even component of the Rost invariant is $0$ & any \\
\end{longtable}

\newpage

\begin{longtable}{c|c|l}
\caption{Conditions on the Tits constructions}\label{tab3} \\
Group     &  Tits construction & Conditions \\
\hline
$2\,{^2\A_2}$   & $\mu_2\times{^2\A_2}$  & scalar is equal\footnote{Every $9$-dimensional Freudenthal algebra has a cohomological invariant of degree $1$ modulo $2$, which is denoted by ${f_1\in H^1(F,\mu_2)}$.} to $f_1$\\
& (scalar\footnote{This scalar is an element in $H^1(F,\mu_2)=F^\times/F^{\times 2}$.} $\times$ $9$-dim Freudenthal algebra) & \\
\hline
$^2\A_5$   & $\mu_2\times\C_3$  & scalar divides\footnote{Every $15$-dimensional Freudenthal algebra has a cohomological invariant of degree $2$ modulo $2$, which is denoted by ${f_2\in H^2(F,\mu_2)}$. This invariant is a pure symbol.} $f_2$ \\
& (scalar $\times$ $15$-dim Freudenthal algebra) & \\
\hline
$^2\A_5$   & $\A_1\times{^2\A_2}$  & quaternions are divisible by $f_1$ \\
& (quaternions $\times$ $9$-dim Freudenthal algebra) & \\
\hline
$^1\D_6$   & $\A_1\times\C_3$  & quaternions divide\footnote{Quaternion algebras correspond pure symbols in $H^2(F,\mu_2)$, where $f_2$ lies. In particular, one can replace ``divide'' by ``equal''.} $f_2$\\
& (quaternions $\times$ $15$-dim Freudenthal algebra) & \\
\hline
$^2\E_6$   & $\mu_2\times\F_4$  & scalar divides\footnote{Every Albert algebra has a cohomological invariant of degree $3$ modulo $2$, which is denoted by ${f_3\in H^3(F,\mu_2)}$. This invariant is a pure symbol.} $f_3$ \\
& (scalar $\times$ Albert) & \\
\hline
$^2\E_6$   & $\G_2\times{^2\A_2}$  & octonions are divisible by $f_1$ \\
& (octonions $\times$ $9$-dim Freudenthal algebra) & \\
\hline
$\E_7$   & $\A_1\times\F_4$ & quaternions divide $f_3$ \\
&  (quaternions $\times$ Albert) & \\
\hline
$\E_7$   & $\C_3\times\G_2$  & $f_2$ divides octonions \\
& ($15$-dim Freudenthal algebra $\times$ octonions) & \\
\hline
  $\E_8$   & $\G_2\times\F_4$  & octonions divide\footnote{Octonion algebras correspond to pure symbols in $H^3(F,\mu_2)$, where $f_3$ lies. In particular, one can replace ``divide'' by ``equal''.} $f_3$\\
  & (octonions $\times$ Albert) & \\
\end{longtable}

\noindent
{\it Comments on Tables~\ref{tab2} and \ref{tab3}}:

\begin{enumerate}[label={\arabic*)},leftmargin=0.6cm]
\item Groups of type $\A_1$ correspond to quaternion algebras, and, thus, have a degree $2$ invariant (the class of the quaternion algebra in the Brauer group of the base field).

\item Groups of type $^2\A_2$ correspond to unitary involutions on central simple algebras of degree $3$ over a quadratic field extension $K$ of the base field $F$. In the present article we investigate cohomological invariants and motives with coefficients in $\zz/2$. Therefore, we can assume that the respective algebra of degree $3$ is split.

In this case groups of type $^2\A_2$ correspond
to $3$-dimensional Hermitian forms. Such forms correspond to $6$-dimensional quadratic forms over $F$, which are divisible by a binary quadratic form given by the quadratic extension $K/F$. It is well known that all such quadratic forms are excellent and have a cohomological invariant of degree $3$.

\item Cohomological invariants modulo $2$ of adjoint groups of type $\C_3$ and more generally of type $\C_n$ with an odd $n$ are well known (see \cite{MD08}). Groups of type $\C_3$ have cohomological invariants $f_2\in H^2(F,\mu_2)$ and $f_4\in H^4(F,\mu_2)$.

\item Groups of type $^2\A_5$ correspond to unitary involutions on central simple algebras of degree $6$ over a quadratic field extension $K$ of the base field $F$. If we assume that our group of type $^2\A_5$ splits over a quadratic field extension of $F$, then it corresponds to a $6$-dimensional Hermitian form on a vector space, which, in turn, corresponds to a $12$-dimensional quadratic form over $F$ divisible by a binary quadratic form given by the quadratic extension $K/F$.

\item It is well known that every Freudenthal algebra of dimension $3+3\cdot 2^i$ (for $i=1$, $2$, $3$) has cohomological invariants $f_i\in H^i(F,\mu_2)$ and $f_{i+2}\in H^{i+2}(F,\mu_2)$ (see \cite[Theorem~40.2]{KMRT}). These algebras correspond to groups of type $\A_2$, $\C_3$ and $\F_4$ resp.

\item The motives of varieties of parabolic subgroups of type $4$ of groups of type $\F_4$ modulo $2$ are computed in \cite[Theorem~3.12]{MD09}. The method of the proof of MacDonald is very elegant and is implicitly related to Cartan's symmetric spaces in the following way.

Taking the symmetric space FII one can apply the algorithm described in \cite[Section~2.5]{GP24} to the variety $X$ of parabolic subgroups of type $4$ of the group of type $\F_4$. It follows that $X$ contains an open subvariety, which is a twisted form of $\Spin_9/\Spin_7$. The latter variety is a $15$-dimensional affine quadric due to a famous exceptional isomorphism. MacDonald indeed shows that the variety $X$ is birationally isomorphic to a $15$-dimensional norm quadric, describes the birational isomorphism explicitly and uses it to decompose the motive of $X$.

\item The degree $5$ invariant and the respective binary motive are constructed in \cite[Theorem~0.2]{GaPe07} and in Proposition~\ref{t211pro}  for groups of type $^2\E_6$, in \cite[Chapter~12]{He23} for groups of type $\E_7$ (see also Proposition~\ref{prop61}), and in \cite[Theorem~8.7]{S16} for groups of type $\E_8$ (see also \cite[Theorem~3.7]{GS10}). Note also that in all cases but $\E_8$ the respective binary motive is a direct summand of a twisted flag variety, which is not generically split or generically quasi-split.
\end{enumerate}

\end{document}